\documentclass{amsart}

\usepackage{amsmath, amsfonts, mathrsfs}
\usepackage{amssymb,latexsym}
\usepackage{enumerate}
\usepackage{bbm}

\usepackage[top=3cm, bottom=3cm, left=3.1cm, right=3.1cm, includefoot, heightrounded]{geometry}

\newtheorem{theorem}{Theorem}[section]

\newtheorem{lemma}[theorem]{Lemma}

\theoremstyle{definition}
\newtheorem{definition}[theorem]{Definition}

\numberwithin{equation}{section}

\begin{document}

\title[The Hake-McShane and Hake-Henstock-Kurzweil integrals]
{The Hake-McShane and Hake-Henstock-Kurzweil integrals over $m$-dimensional unbounded sets}

\author{ Sokol Bush Kaliaj }

\address{
Mathematics Department, 
Science Natural Faculty, 
University of Elbasan,
Elbasan, 
Albania.
}

\email{sokolkaliaj@yahoo.com}

\thanks{}



\begin{abstract} 
In this paper, we extend the Hake-McShane and Hake-Henstock-Kurzweil integrals 
of Banach space valued functions  from $m$-dimensional open and bounded sets to 
$m$-dimensional sets $G$ such that $|G \setminus G^{o}|=0$. 
We will prove the full descriptive characterizations of 
new integrals in terms of the locally McShane and locally Henstock-Kurzweil integrals. 
\end{abstract}

\maketitle

\section{Introduction and Preliminaries}

In this paper, we continue the investigation of characterizations of 
the Hake-McShane and Hake-Henstock-Kurzweil integrals 
in terms of the McShane and Henstock-Kurzweil integrals,   
started in \cite{KAL1}. 
At first, we define the Hake-McShane and Hake-Henstock-Kurzweil integrals 
of Banach space valued functions 
defined on a subset $E \subset \mathbb{R}^{m}$ 
such that $|E \setminus E^{o}|=0$, 
see Definition \ref{D_Hakeint}. 
If $E$ is a bounded set and $E = E^{o}$, 
then Definition \ref{D_Hakeint} is the same with the corresponding  definition in   \cite{KAL1}.
In the paper \cite{KAL1} are proved full descriptive characterizations of 
the Hake-McShane and Hake-Henstock-Kurzweil integrals 
of Banach space valued functions defined on a bounded and open subset $G \subset \mathbb{R}^{m}$ 
in terms of 
the McShane and Henstock-Kurzweil integrals, 
see Theorems 3.1 and 3.2 in \cite{KAL1}.
Here, we will prove full descriptive characterizations of 
the Hake-McShane and Hake-Henstock-Kurzweil integrals 
of Banach space valued functions defined on a subset $G \subset \mathbb{R}^{m}$ 
such that $|G \setminus G^{o}|=0$ 
in terms of the locally McShane and locally Henstock-Kurzweil integrals,  
see Theorems \ref{T2.1} and \ref{T2.2}.

Throughout this paper $X$ denotes a real Banach space with 
the norm $||\cdot||$. 
The Euclidean space $\mathbb{R}^{m}$ is equipped with the maximum norm. 
$B_{m}(t,r)$ denotes the open ball in $\mathbb{R}^{m}$ with center $t$ 
and radius $r > 0$. 
We denote by $\mathcal{L}(\mathbb{R}^{m})$ the $\sigma$-algebra of Lebesgue measurable subsets 
of $\mathbb{R}^{m}$ and by 
$\lambda$ the Lebesgue measure on $\mathcal{L}(\mathbb{R}^{m})$.  
$|A|$ denotes the Lebesgue measure of $A \in \mathcal{L}(\mathbb{R}^{m})$.

The subset $\prod_{j=1}^{m} [a_{j},b_{j}] \subset \mathbb{R}^{m}$ 
is said to be a \textit{closed non-degenerate interval} in $\mathbb{R}^{m}$, 
if $-\infty < a_{j} < b_{j} < +\infty$,  for $j=1, \dotsc, m$.
Two closed non-degenerate intervals $I$ and $J$ in $\mathbb{R}^{m}$ are said to be 
\textit{non-overlapping} if 
$I^{o} \cap J^{o} = \emptyset$, where $I^{o}$ denotes the \textit{interior} of $I$. 
By $\mathcal{I}$ the family  of all closed non-degenerate subintervals in $\mathbb{R}^{m}$ is denoted 
and by $\mathcal{I}_{E}$ the family  of all closed non-degenerate 
subintervals in $E \subset \mathbb{R}^{m}$.   
A  function $F: \mathcal{I}_{E} \to X$  
is said to be an \textit{additive interval function}, 
if for each two non-overlapping intervals $I, J \in \mathcal{I}_{E}$ such that 
$I \cup J \in \mathcal{I}_{E}$, we have
\begin{equation*}
F(I \cup J) = F(I) + F(J).
\end{equation*}
A pair $(t, I)$ of a point $t \in E$  and an interval $I \in \mathcal{I}_{E}$ 
is called an \textit{$\mathcal{M}$-tagged interval} in $E$, $t$ is the tag of $I$. 
Requiring $t \in I$ for the tag of $I$ 
we get the concept of an \textit{$\mathcal{HK}$-tagged interval} in $E$. 
A finite collection $\{ (t_{i}, I_{i}) : i = 1, \dotsc, p \}$ of 
$\mathcal{M}$-tagged intervals ($\mathcal{HK}$-tagged intervals) in $E$ 
is called an  \textit{$\mathcal{M}$-partition} (\textit{$\mathcal{HK}$-partition}) in $E$, 
if $\{ I_{i} : i = 1, \dotsc, p \}$ is a collection of pairwise non-overlapping intervals 
in $\mathcal{I}_{E}$. 
Given $Z \subset E$, a positive function $\delta: Z \to (0,+\infty)$ 
is called a \textit{gauge} on $Z$. 
We say that an 
$\mathcal{M}$-partition ($\mathcal{HK}$-partition) $\pi = \{ (t_{i}, I_{i}) : i = 1, \dotsc, p \}$ 
in $E$ is 
\begin{itemize}
\item
$\mathcal{M}$-partition ($\mathcal{HK}$-partition) of $E$, if $\cup_{i=1}^{p} I_{i} = E$,
\item
$Z$-tagged if $\{t_{1}, \dotsc, t_{p} \} \subset Z$,
\item 
$\delta$-fine if for each $i = 1, \dotsc, p$, we have 
$
I_{i} \subset B_{m}(t_{i},\delta(t_{i})). 
$
\end{itemize}

We now recall the definitions of the McShane and Henstock-Kurzweil integrals of a function 
$f:J \to X$, where $J$ is a fixed interval in $\mathcal{I}$. 
The function $f$ is said to be 
\textit{McShane (Henstock-Kurzweil) integrable} on $J$ 
if there is a vector 
$x_{f} \in X$ such that for every $\varepsilon>0$, 
there exists a gauge $\delta$ on $J$  
such that for every $\delta$-fine $\mathcal{M}$-partition ($\mathcal{HK}$-partition) 
$\pi$ of $J$, we have 
$$
||\sum_{(t,I) \in \pi} f(t)|I| - x_{f}|| < \varepsilon.
$$
In this case, the vector $x_{f}$ is said to be the  
\textit{McShane (Henstock-Kurzweil) integral} of $f$ on $J$ 
and we set $x_{f}=(M)\int_{J} f d \lambda$ ($x_{f}=(HK)\int_{J} f d \lambda$). 
The function $f$ is said to be 
\textit{McShane (Henstock-Kurzweil) integrable} over a subset $A \subset J$, 
if the function $f . \mathbbm{1}_{A} : J \to X$ is McShane (Henstock-Kurzweil) integrable on $J$, 
where $\mathbbm{1}_{A}$ is the characteristic function of the set $A$. 
The McShane (Henstock-Kurzweil) integral of $f$ over $A$ will be denoted by $(M)\int_{A}f d\lambda$ 
($(HK)\int_{A}f d\lambda$).
If $f:J \to X$ is McShane integrable on $J$, then by Theorem 4.1.6 in \cite{SCH1} 
the function $f$ is the McShane integrable on each Lebesgue measurable subset $A \subset J$, 
while by Theorem 3.3.4 in \cite{SCH1}, 
if $f$ is Henstock-Kurzweil integrable on $J$, then $f$  
is the Henstock-Kurzweil integrable on each $I \in \mathcal{I}_{J}$. 
Therefore, we can define an additive interval function 
$F: \mathcal{I}_{J} \to X$ as follows 
$$
F(I) = (M) \int_{I} f d \lambda,~  (~ F(I) = (HK) \int_{I} f d \lambda ~),~
\text{ for all }I \in \mathcal{I}_{J},
$$
which is called the primitive of $f$.

The basic properties of the McShane integral and the Henstock-Kurzweil integral 
can be found in \cite{BON}, \cite{CAO}, \cite{DIP}, \cite{FRE1}-\cite{FRE3}, \cite{GORD1}-\cite{GORD3},  
\cite{GUO1}, \cite{GUO2}, 
\cite{KURZ}, \cite{LEE1}, \cite{LEE2} and \cite{SCH1}. 
We do not present them here. 
The reader is referred to the above mentioned references for the details.

\begin{definition}\label{D_LocInt}
Assume that an open subset $W \subset \mathbb{R}^{m}$,   
a function $f : W \to X$ and an additive interval function $F: \mathcal{I}_{W} \to X$ 
are given. 
For each $J \in \mathcal{I}_{W}$, we denote
$$
f_{J} = f \vert_{J}\text{ and }
F_{J} = F \vert_{\mathcal{I}_{J}}. 
$$ 
The function $f$ is said to be 
\textit{locally McShane (locally Henstock-Kurzweil) integrable} on $W$ with the primitive $F$, 
if for each $J \in \mathcal{I}_{W}$,  $f_{J} = f \vert_{J}$ 
is McShane (Henstock-Kurzweil) integrable on $J$ 
with the primitive 
$F_{J} = F \vert_{\mathcal{I}_{J}}$.
\end{definition}

We now fix a subset $E \subset \mathbb{R}^{m}$ such that 
$|E \setminus E^{o}| = 0$, where $E^{o}$ is the interior of $E$.   
A sequence $(I_{k})$ of pairwise non-overlapping intervals in $\mathcal{I}_{E}$ is said to be a 
\textit{division of} $E^{o}$, if  
$$
E^{o} = \bigcup_{k=1}^{+\infty} I_{k}.
$$   
We denote by $\mathscr{D}_{E^{o}}$ the family of all divisions of $E^{o}$.   
By Lemma 2.43 in \cite{FOLL}, the family $\mathscr{D}_{E^{o}}$ is not empty.

\begin{definition}\label{D_Hake}
An additive interval function  $F : \mathcal{I}_{E} \to X$  
is said to be a \textit{Hake-function}, 
if given a division $(I_{k}) \in \mathscr{D}_{E^{o}}$, we have
\begin{itemize}
\item
the series 
$$
\sum_{k: |I \cap I_{k}| > 0 } F(I \cap I_{k})
$$ 
is unconditionally convergent in $X$, 
for each $I \in \mathcal{I}$,
\item
the equality 
$$
F(I) = \sum_{k: |I \cap I_{k}| > 0 } F(I \cap I_{k}),
$$ 
holds for all $I \in \mathcal{I}_{E}$. 
\end{itemize} 
\end{definition}

\begin{definition}\label{D_NegV}
We say that the additive interval function $F : \mathcal{I}_{E} \to X$ 
has \textit{$\mathcal{M}$-negligible ($\mathcal{HK}$-negligible) variation over} 
a subset $Z \subset \mathbb{R}^{m}$, 
if for each $\varepsilon >0$ there exists a gauge 
$\delta_{v}$ on $Z$ such that 
for each $Z$-tagged $\delta_{v}$-fine $\mathcal{M}$-partition ($\mathcal{HK}$-partition) 
$\pi_{v}$ in $\mathbb{R}^{m}$, we have
\begin{itemize}
\item
the series 
$$
\sum_{k: |I \cap I_{k}| > 0 } F(I \cap I_{k})
$$ 
is unconditionally convergent in $X$, for each $(t,I) \in \pi_{v}$, 
\item
the inequality 
$$
||\sum_{(t,I) \in \pi_{v}} 
\left (
\sum_{k: |I \cap I_{k}| > 0 } F(I \cap I_{k})
\right ) 
|| < \varepsilon,
$$
holds, 
\end{itemize}
whenever $(I_{k}) \in \mathscr{D}_{E^{o}}$. 
We say that $F$ has  
\textit{$\mathcal{M}$-negligible ($\mathcal{HK}$-negligible) variation outside} of $E^{o}$ 
if $F$ has $\mathcal{M}$-negligible ($\mathcal{HK}$-negligible) variation 
over $(E^{o})^{c} = \mathbb{R}^{m} \setminus E^{o}$.
\end{definition}


\begin{definition}\label{D_Hakeint}
We say that a function $f: E \to X$ is \textit{Hake-McShane (Hake-Henstock-Kurzweil) integrable} 
on $E$ with 
the primitive $F: \mathcal{I}_{E} \to X$,  
if we have
\begin{itemize}
\item
for each $\varepsilon >0$ there exists a gauge $\delta_{\varepsilon}$ on $E^{o}$ such that  
for each $\delta_{\varepsilon}$-fine $\mathcal{M}$-partition ($\mathcal{HK}$-partition) $\pi$ in 
$E^{o}$, 
we have
$$
|| \sum_{(t,I) \in \pi } 
(~ f(t)|I| - F(I)~) || < \varepsilon,   
$$
\item
$F$ is a Hake-function,
\item
$F$ has $\mathcal{M}$-negligible ($\mathcal{HK}$-negligible) variation outside of $E^{o}$.
\end{itemize}
\end{definition}
Clearly, if $E$ is a bounded set and $E = E^{o}$, then 
Definition \ref{D_Hakeint} is the same with the corresponding definition in \cite{KAL1}.

\section{The Main results}

From now on $G$ will be a subset of  $\mathbb{R}^{m}$ such that 
$G^{o} \neq \emptyset$ and $|G \setminus G^{o}|=0$. 
The main results are Theorems \ref{T2.1} and \ref{T2.2}. 
Let us start with a few auxiliary lemmas.

\begin{lemma}\label{L2.1}
Let $f: G \to X$ be a function and let $J \in \mathcal{I}$. 
Then, given $\varepsilon >0$ there exists a gauge $\delta$ on $Z = J \cap (G \setminus G^{o})$ 
such that for each $\delta$-fine $Z$-tagged $\mathcal{M}$-partition $\pi$ in $J$, we have
$$
|| \sum_{(t,I) \in \pi} f(t)|I| ~|| < \varepsilon.
$$ 
\end{lemma}
\begin{proof}
Define a function $g_{J} : J \to X$ as follows
\begin{equation*}
g_{J}(t) = 
\left \{
\begin{array}{ll}
f(t) & \text{ if }t \in Z \\
0 & \text{ otherwise }. 
\end{array}
\right.
\end{equation*} 
Then, by Theorem 3.3.1 in \cite{SCH1}, 
$g_{J}$ is McShane  integrable on $J$ and
$$
(M)\int_{I} g_{J} d \lambda = 0, \text{ for all }I \in \mathcal{I}_{J}. 
$$
Therefore, by Lemma 3.4.2 in \cite{SCH1}, 
given $\varepsilon >0$ there exists a gauge $\delta$ on $Z$  
such that for each $\delta$-fine $Z$-tagged $\mathcal{M}$-partition $\pi$ in $J$ we have
$$
|| \sum_{(t,I) \in \pi} g_{J}(t)|I| ~|| < \varepsilon,
$$ 
and since 
$g_{J}(t) = f(t)$ for all $t \in Z$, the last result proves the lemma.
\end{proof}

\begin{lemma}\label{L2.2}
Let $f: G \to X$ be a function, 
let $F: \mathcal{I}_{G} \to X$ be an additive interval function and let $J \in \mathcal{I}$.  
If $F$ has $\mathcal{M}$-negligible variation outside of $G^{o}$, then
given $\varepsilon >0$ there exists a gauge $\delta$ on $Z = J \cap (G \setminus G^{o})$ 
such that for each $\delta$-fine $Z$-tagged $\mathcal{M}$-partition $\pi$ in $J$ we have
$$
|| \sum_{(t,I) \in \pi} 
\left ( ~f(t)|I| - \sum_{k:|I \cap I_{k}| >0} F(I \cap I_{k})~ \right ) 
|| < \varepsilon,
$$ 
whenever $(I_{k}) \in \mathscr{D}_{G^{o}}$. 
\end{lemma}
\begin{proof}
Since $F$ has $\mathcal{M}$-negligible variation  
outside of $G^{o}$, 
given $\varepsilon >0$ there exists a gauge $\delta_{v}$ 
on $(G^{o})^{c}$  
such that for each $\delta_{v}$-fine $(G^{o})^{c}$-tagged $\mathcal{M}$-partition 
$\pi_{v}$ in $J$, 
we have
$$
|| \sum_{(t,I) \in \pi_{v}} 
\left ( \sum_{k:|I \cap I_{k}| >0} F(I \cap I_{k})~ \right ) 
|| < \frac{\varepsilon}{2},
$$ 
whenever $(I_{k}) \in \mathscr{D}_{G^{o}}$. 

By Lemma \ref{L2.1}, 
there exists a gauge $\delta_{0}$ on $Z$ 
such that for each $\delta_{0}$-fine $Z$-tagged $\mathcal{M}$-partition 
$\pi$ in $J$, we have
$$
|| \sum_{(t,I) \in \pi} f(t)|I| ~|| < \frac{\varepsilon}{2}.
$$ 
Define a gauge $\delta$ on $Z$ by $\delta(t) = \min\{\delta_{v}(t), \delta_{0}(t)\}$ for all $t \in Z$. 
Let $\pi$ be a $\delta$-fine $Z$-tagged $\mathcal{M}$-partition $\pi$ in $J$. 
Then,
\begin{equation*}
\begin{split}
|| \sum_{(t,I) \in \pi} 
\left ( ~f(t)|I| - \sum_{k:|I \cap I_{k}| >0} F(I \cap I_{k})~ \right ) 
||
& \leq	
|| \sum_{(t,I) \in \pi} f(t)|I| ~|| \\
&+
|| \sum_{(t,I) \in \pi} 
\left (\sum_{k:|I \cap I_{k}| >0} F(I \cap I_{k})~ \right ) 
|| < \frac{\varepsilon}{2} + \frac{\varepsilon}{2} = \varepsilon,
\end{split}
\end{equation*}
and this ends the proof.
\end{proof}

The next lemma can be proved in the same way as Lemma \ref{L2.2}. 

\begin{lemma}\label{L2.3}
Let $f: G \to X$ be a function, 
let $F: \mathcal{I}_{G} \to X$ be an additive interval function and let $J \in \mathcal{I}$.  
If $F$ has $\mathcal{HK}$-negligible variation outside of $G^{o}$, then
given $\varepsilon >0$ there exists a gauge $\delta$ on $Z = J \cap (G \setminus G^{o})$ 
such that for each $\delta$-fine $Z$-tagged $\mathcal{HK}$-partition $\pi$ in $J$ we have
$$
|| \sum_{(t,I) \in \pi} 
\left ( ~f(t)|I| - \sum_{k:|I \cap I_{k}| >0} F(I \cap I_{k})~ \right ) 
|| < \varepsilon,
$$ 
whenever $(I_{k}) \in \mathscr{D}_{G^{o}}$. 
\end{lemma}

\begin{lemma}\label{L2.4} 
Let $W$ be an open subset of $\mathbb{R}^{m}$, 
let $v: W \to X$ be a function and let $V: \mathcal{I}_{W} \to X$ be an additive interval function. 
Then, the following statements are equivalent:
\begin{itemize}
\item[(i)]
$v$ is locally McShane integrable on $W$ with the primitive $V$,
\item[(ii)] 
for each $\varepsilon >0$ there exists a gauge $\delta_{\varepsilon}$ on $W$ such that  
for each $\delta_{\varepsilon}$-fine $\mathcal{M}$-partition $\pi$ in 
$W$, 
we have
$$
|| \sum_{(t,I) \in \pi } 
(~ v(t)|I| - V(I)~) || < \varepsilon. 
$$
\end{itemize}
\end{lemma}
\begin{proof}
$(i) \Rightarrow (ii)$ 
Assume that $v$ is locally McShane integrable on $W$ with the primitive $V$. 
Fix a division $(C_{k})$ of $W$ and let $\varepsilon >0$ be given. 
For each $k \in \mathbb{N}$, 
we denote 
$$ 
v_{k} = v \vert_{C_{k}} \text{ and } V_{k} = V \vert_{\mathcal{I}_{C_{k}}}.
$$ 
By hypothesis each function $v_{k}$ is McShane integrable on $C_{k}$ 
with the primitive $V_{k}$. 
Hence, by Lemma 3.4.2 in \cite{SCH1}, 
there exists a gauge $\delta_{k}$ on $C_{k}$ such that 
for each $\delta_{k}$-fine $\mathcal{M}$-partition $\pi_{k}$ in $C_{k}$, we have
\begin{equation}\label{eq_SaksHen}
||\sum_{(t,I) \in \pi_{k}} (~ v_{k}(t)|I| - V_{k}(I)~) || \leq 
\frac{1}{2^{k}}\frac{\varepsilon}{2}.  
\end{equation}

Note that for any $t \in  W = \cup_{k} C_{k}$, we have the following possible cases:
\begin{itemize}
\item
there exists $i_{0} \in \mathbb{N}$ such that $t \in (C_{i_{0}})^{o}$, 
\item
there exists $j_{0} \in \mathbb{N}$ such that $t \in C_{j_{0}} \setminus (C_{j_{0}})^{o}$. 
In this case, 
there exists a finite set 
$\mathcal{N}_{t} = \{j \in \mathbb{N}: t \in C_{j} \setminus (C_{j})^{o} \}$ such that 
$t \in \bigcap_{j \in \mathcal{N}_{t}} C_{j}$ and $t \notin C_{k}$,  
for all $k \in \mathbb{N} \setminus \mathcal{N}_{t}$. 
Hence, $t \in ( \cup_{j \in \mathcal{N}_{t}} C_{j} )^{o}$. 
\end{itemize} 
For each $k \in \mathbb{N}$, we can choose $\delta_{k}$ so that for any $t \in W$, we have
$$
t \in (C_{k})^{0} \Rightarrow B_{m}(t, \delta_{k}(t)) \subset C_{k}
$$
and
$$
t \in C_{k} \setminus (C_{k})^{o} \Rightarrow  
B_{m}(t, \delta_{k}(t)) \subset \bigcup_{j \in \mathcal{N}_{t}} C_{j}.
$$
Define a gauge $\delta_{\varepsilon} : W \to (0,+\infty)$ as follows.  
For each $t \in W$, we choose 
$$
\delta_{\varepsilon}(t) 
= 
\left \{
\begin{array}{ll}
\delta_{i_{0}}(t) & \text{ if }t \in (C_{i_{0}})^{o}  \\
\min \{ \delta_{j}(t): j \in \mathcal{N}_{t} \} & \text{ otherwise}.
\end{array}
\right.
$$
Let $\pi$ be an arbitrary $\delta_{\varepsilon}$-fine $\mathcal{M}$-partition in $W$. 
Then, $\pi = \pi_{1} \cup \pi_{2}$, where 
\begin{equation*}
\begin{split}
\pi_{1} 
&= \{ (t,I) \in \pi : (\exists i_{0} \in \mathbb{N})[t \in (C_{i_{0}})^{o}] \} \\
\pi_{2} 
&= \{(t,I) \in \pi : 
(\exists j_{0} \in \mathbb{N})[t \in C_{j_{0}} \setminus (C_{j_{0}})^{o}] \}.
\end{split}
\end{equation*}
Hence,    
\begin{equation}\label{eq_SS.1}
\begin{split}
||\sum_{(t,I) \in \pi} (~v(t)|I| - V(I)~) || 
\leq&
||\sum_{(t,I) \in \pi_{1}} (~v(t)|I| - V(I)~) || \\
+& 
||\sum_{(t,I) \in \pi_{2}} (~v(t)|I| - V(I)~) || 
\end{split}
\end{equation} 
Note that, if we define
\begin{equation*}
\begin{split}
\pi_{1}^{k} &= \{( t,I) \in \pi_{1}: t \in (C_{k})^{o} \}, \\ 
\pi_{2}^{k} &= \{ (t, I \cap C_{k}) : (t,I) \in \pi_{2}, 
t \in C_{k} \setminus (C_{k})^{o}, |I \cap C_{k}|>0 \},
\end{split}
\end{equation*}
then $\pi_{1}^{k}$ and $\pi_{2}^{k}$ are $\delta_{k}$-fine $\mathcal{M}$-partitions in $C_{k}$. 
Therefore, by \eqref{eq_SaksHen}, it follows that
\begin{equation*}
\begin{split}
||\sum_{(t,I) \in \pi_{1}} (~v(t)|I| - V(I)~) ||   
&=
||\sum_{k} 
\left (
\sum_{\underset{t \in (C_{k})^{0}}{(t,I) \in \pi_{1}}}  
(~v(t)|I| - V(I)~)
\right )
|| \\
&\leq
\sum_{k} ||
\left ( \sum_{(t,I) \in \pi_{1}^{k}} 
(~v_{k}(t)|I| - V_{k}(I)~) 
\right )
|| \\
&\leq
\sum_{k=1}^{+\infty} \frac{1}{2^{k}}\frac{\varepsilon}{2} = \frac{\varepsilon}{2}
\end{split}
\end{equation*}
and  
\begin{equation*}
\begin{split}
||\sum_{(t,I) \in \pi_{2}} (~v(t)|I| - V(I)~) || 
&= 
||\sum_{(t,I) \in \pi_{2}} 
\left ( \sum_{\underset{| I \cap C_{j}| >0}{j \in \mathcal{N}_{t}}} 
(~ v(t) . | I \cap C_{j}| - V(I \cap C_{j}) ~) \right ) || \\
=& 
||\sum_{(t,I) \in \pi_{2}} 
\left ( \sum_{\underset{| I \cap C_{j}| >0}{j \in \mathcal{N}_{t}}}
(~v_{j}(t) . | I \cap C_{j}| - V_{j}(I \cap C_{j})~ ) \right ) || \\
=&
||\sum_{k} 
\left ( 
\sum_{(t,I \cap C_{k}) \in \pi_{2}^{k}} 
(~v_{k}(t) . |I \cap C_{k}| - V_{k}(I \cap C_{k})~)  
\right ) ||\\
\leq&
\sum_{k} 
||
\left ( 
\sum_{(t, I \cap C_{k}) \in \pi_{2}^{k}} 
(~v_{k}(t) . | I \cap C_{k} | - V_{k}( I \cap C_{k} )~)  
\right ) || \\
\leq&
\sum_{k=1}^{+\infty} \frac{1}{2^{k}}\frac{\varepsilon}{2} = \frac{\varepsilon}{2}.
\end{split}
\end{equation*}
The last results together with \eqref{eq_SS.1} yield
$$
||\sum_{(t,I) \in \pi} (~v(t) . |I| - V(I)~)|| < \varepsilon.
$$

$(ii) \Rightarrow (i)$ 
Assume that $(ii)$ holds.   
Then, given $\varepsilon >0$ there exists a gauge $\delta_{\varepsilon}$ on $W$ such that 
for each $\delta_{\varepsilon}$-fine $\mathcal{M}$-partition $\pi$ in $W$, we have
\begin{equation*}
|| \sum_{(t,I) \in \pi } (~ v(t)|I| - V(I)~) || < \varepsilon.   
\end{equation*}

Fix an arbitrary   $J \in \mathcal{I}_{W}$. 
We will prove that $v_{J} = v \vert_{J}$   
is McShane integrable on $J$ with the primitive $V_{J} = V \vert_{\mathcal{I}_{J}}$.  
Let $\pi_{J}$ be a $\delta_{J}$-fine $\mathcal{M}$-partition of $J$, 
where $\delta_{J} = \delta_{\varepsilon} \vert_{J}$. 
Then, $\pi_{J}$ is a $\delta_{\varepsilon}$-fine $\mathcal{M}$-partition in $W$
and, therefore, 
\begin{equation*}
\begin{split}
||\sum_{(t,I) \in \pi_{J}} (~  v_{J}(t). |I|-V_{J}(I) ~)|| 
&= || \sum_{(t,I) \in \pi_{J}} (~ v(t). |I| - V(I) ~) || < \varepsilon. 
\end{split} 
\end{equation*}
This means that $v_{J}$ is McShane integrable on $J$ with the primitive $V_{J}$. 
Since $J$ was arbitrary, the last result means that  
$v$ is locally McShane integrable on $W$ with the primitive $V$, 
and this ends the proof. 
\end{proof}

Using Lemma 3.4.1 in \cite{SCH1}, 
the next lemma can be proved in the same manner 
as Lemma \ref{L2.4}.

\begin{lemma}\label{L22.4} 
Let $W$ be an open subset of $\mathbb{R}^{m}$, 
let $v: W \to X$ be a function and let $V: \mathcal{I}_{W} \to X$ be an additive interval function. 
Then, the following statements are equivalent:
\begin{itemize}
\item[(i)]
$v$ is locally Henstock-Kurzweil integrable on $W$ with the primitive $V$,
\item[(ii)] 
for each $\varepsilon >0$ there exists a gauge $\delta_{\varepsilon}$ on $W$ such that  
for each $\delta_{\varepsilon}$-fine $\mathcal{HK}$-partition $\pi$ in 
$W$, 
we have
$$
|| \sum_{(t,I) \in \pi } 
(~ v(t)|I| - V(I)~) || < \varepsilon. 
$$
\end{itemize}
\end{lemma}

Given a function $f:G \to X$, 
we denote by $h: \mathbb{R}^{m} \to X$ the function defined as follows
\begin{equation*}
h(t) = 
\left \{
\begin{array}{ll}
f(t) & \text{ if }t \in G \\
0 & \text{ if }t \in \mathbb{R}^{m} \setminus G. 
\end{array}
\right.
\end{equation*}

\begin{lemma}\label{L2.5}
Let $f:G \to X$ be a function. 
If $h$ is locally McShane integrable on $\mathbb{R}^{m}$ with the primitive $H$, 
then given $\varepsilon >0$ there exists a gauge $\delta_{v}$ on $(G^{o})^{c}$ 
such that for each 
$\delta_{v}$-fine  $(G^{o})^{c}$-tagged $\mathcal{M}$-partition $\pi_{v}$ in $\mathbb{R}^{m}$, 
we have
\begin{equation*}
||\sum_{(t,I) \in \pi_{v}} H(I) || < \varepsilon.
\end{equation*}
\end{lemma}
\begin{proof}
By Lemma \ref{L2.4}, 
given $\varepsilon >0$ there exists a gauge $\delta_{\varepsilon}$ 
on $\mathbb{R}^{m}$ such that  
for each $\delta_{\varepsilon}$-fine $\mathcal{M}$-partition $\pi$ in 
$\mathbb{R}^{m}$, 
we have
\begin{equation}\label{eq_C21.1}
|| \sum_{(t,I) \in \pi } 
(~ h(t)|I| - H(I)~) || < \frac{\varepsilon}{3}. 
\end{equation}

For each $k \in \mathbb{N}$, let
$$
N_{k} = \{ t \in G \setminus G^{o} :  k-1 \leq ||h(t)|| <k \}.
$$
Since $|G \setminus G^{o}| = 0$, for each $k \in \mathbb{N}$, 
we have $|N_{k}| = 0$ and, therefore, 
there exists an open set $G_{k}$ such that 
\begin{equation}\label{eq_C21.2}
G_{k} \supset N_{k} 
\quad\text{and}\quad
|G_{k}| < \frac{\varepsilon}{3 .k .2^{k}}.
\end{equation}
Define a gauge $\delta_{v}$ on $(G^{o})^{c}$ in such a way that 
$$
t \in G^{c} \Rightarrow \delta_{v}(t) = \delta_{\varepsilon}(t)
$$
and
$$
t \in N_{k} \subset G_{k} \Rightarrow  
B_{m}(t, \delta_{v}(t)) \subset G_{k}\text{ and }\delta_{v}(t)) \leq \delta_{\varepsilon}(t). 
$$
Suppose that $\pi_{v}$ is an arbitrary 
$\delta_{v}$-fine $(G^{o})^{c}$-tagged $\mathcal{M}$-partition in $\mathbb{R}^{m}$. 
Then, $\pi_{v} = \pi_{v}^{1} \cup \pi_{v}^{2}$ where
$$
\pi_{v}^{1} = \{ (t,I) \in \pi_{v}: t \in G^{c}\}
\text{ and }
\pi_{v}^{2} = \{ (t,I) \in \pi_{v}: t \in G \setminus G^{o} \}.
$$
Hence,
\begin{equation}\label{eq_C21.3}
\begin{split}
||\sum_{(t,I) \in \pi_{v}} H(I) || \leq 
||\sum_{(t,I) \in \pi_{v}^{1}} H(I) || + ||\sum_{(t,I) \in \pi_{v}^{2}} H(I) ||
\end{split} 
\end{equation} 
Since $\pi_{v}^{1}$ and $\pi_{v}^{2}$ are also 
$\delta_{\varepsilon}$-fine $\mathcal{M}$-partitions in $\mathbb{R}^{m}$, 
we obtain by \eqref {eq_C21.1} that 
\begin{equation}\label{eq_C21.4}
\begin{split}
||\sum_{(t,I) \in \pi_{v}^{1}} H(I) || = 
||\sum_{(t,I) \in \pi_{v}^{1}} (~ h(t)|I| - H(I)~)||< \frac{\varepsilon}{3}
\end{split} 
\end{equation} 
and 
\begin{equation}\label{eq_C21.5}
\begin{split}
||\sum_{(t,I) \in \pi_{v}^{2}} H(I) || 
&\leq
||\sum_{(t,I) \in \pi_{v}^{2}} (~ h(t)|I| - H(I)~)|| + 
||\sum_{(t,I) \in \pi_{v}^{2}} h(t)|I| ~|| \\ 
&< \frac{\varepsilon}{3} + ||\sum_{(t,I) \in \pi_{v}^{2}} h(t)|I| ~|| 
=\frac{\varepsilon}{3} + ||\sum_{(t,I) \in \pi_{v}^{2}} f(t)|I| ~||. 
\end{split} 
\end{equation} 
By \eqref{eq_C21.2}, we have also
\begin{equation*}
\begin{split}
||\sum_{(t,I) \in \pi_{v}^{2}} h(t)|I| ~|| 
=& ||\sum_{k=1}^{+\infty} 
\left ( \sum_{\underset{t \in N_{k}}{(t,I) \in \pi_{v}^{2}}} h(t)|I| \right) ~|| \\
\leq&
\sum_{k=1}^{+\infty} 
||\left ( \sum_{\underset{t \in N_{k}}{(t,I) \in \pi_{v}^{2}}} h(t)|I| \right) || 
<
\sum_{k=1}^{+\infty} 
k\left ( \sum_{\underset{t \in N_{k}}{(t,I) \in \pi_{v}^{2}}} |I| \right) || \\
\leq&
\sum_{k=1}^{+\infty} k |G_{k}| < 
\sum_{k=1}^{+\infty} \frac{\varepsilon}{3 . 2^{k}} = \frac{\varepsilon}{3}. 
\end{split} 
\end{equation*} 
The last result together with \eqref{eq_C21.3}, \eqref{eq_C21.4} 
and \eqref{eq_C21.5} yields 
\begin{equation*}
\begin{split}
||\sum_{(t,I) \in \pi_{v}} H(I) || < \frac{\varepsilon}{3} + \frac{\varepsilon}{3} + \frac{\varepsilon}{3} 
= \varepsilon,
\end{split} 
\end{equation*}
and since $\pi_{v}$ was arbitrary, 
this ends the proof.
\end{proof}

Using Lemma \ref{L22.4}, the next lemma can be proved in the same manner 
as Lemma \ref{L2.5}.  

\begin{lemma}\label{L2.55}
Let $f:G \to X$ be a function. 
If $h$ is locally Henstock-Kurzweil integrable on $\mathbb{R}^{m}$ with the primitive $H$, 
then given $\varepsilon >0$ there exists a gauge $\delta_{v}$ on $(G^{o})^{c}$ 
such that for each 
$\delta_{v}$-fine  $(G^{o})^{c}$-tagged $\mathcal{HK}$-partition $\pi_{v}$ in $\mathbb{R}^{m}$, 
we have
\begin{equation*}
||\sum_{(t,I) \in \pi_{v}} H(I) || < \varepsilon.
\end{equation*}
\end{lemma}

We are now ready to present the main results.

\begin{theorem}\label{T2.1}
Let $f: G \to X$ be a function and let $F: \mathcal{I}_{G} \to X$ be an additive interval function.  
Then, the following statements are equivalent:
\begin{itemize}
\item[(i)]
$f$ is Hake-McShane integrable on $G$ with the primitive $F$,
\item[(ii)] 
$h$ is locally McShane integrable on $\mathbb{R}^{m}$ with the primitive $H$ 
such that $H(I) = F(I)$, for all $I \in \mathcal{I}_{G}$, and
\begin{equation}\label{eq_T21.0}
H(I) = \sum_{k:|I \cap C_{k}|>0} F(I \cap C_{k}), \text{ for all }I \in \mathcal{I},
\end{equation}
whenever $(C_{k}) \in \mathscr{D}_{G^{o}}$. 
\end{itemize}
\end{theorem}
\begin{proof}
$(i) \Rightarrow (ii)$ Assume that $f$ is Hake-McShane integrable on $G$ 
with the primitive $F$. 
Then, 
by Definition \ref{D_Hakeint}, 
given $\varepsilon >0$ there exists a gauge $\delta_{0}$ on $G^{o}$ such that 
for each $\delta_{0}$-fine $\mathcal{M}$-partition $\pi_{0}$ in $G^{o}$, we have
\begin{equation*}
|| \sum_{(t,I) \in \pi_{0} } (~ f(t)|I| - F(I)~) || < \frac{\varepsilon}{4}.   
\end{equation*}

Since $F$ is a Hake-function, we can define an additive interval function 
$H: \mathcal{I} \to X$ as follows
\begin{equation*}
H(I) = \sum_{k: |I \cap C_{k}|>0} F(I \cap C_{k}), 
\text{ for all } I \in \mathcal{I},
\end{equation*} 
where $(C_{k})$ is an arbitrary division of $G^{o}$. 
Clearly, 
$H(I) = F(I)$, for all $I \in \mathcal{I}_{G}$. 
We can choose $\delta_{0}$ so that 
$B_{m}(t, \delta_{0}(t)) \subset G^{o}$ for all $t \in G^{o}$.  
Hence, 
\begin{equation}\label{eq_T21.1} 
|| \sum_{(t,I) \in \pi_{0} } (~ h(t)|I| - H(I)~) || < \frac{\varepsilon}{4},   
\end{equation}
whenever $\pi_{0}$ is a 
$\delta_{0}$-fine $\mathcal{M}$-partition in $G^{o}$.

Assume that an interval $J \in \mathcal{I}$ is given. 
We are going to prove that $h_{J} = h \vert_{J}$ is McShane integrable on $J$ 
with the primitive $H_{J} = H \vert_{\mathcal{I}_{J}}$.

Since $F$ has $\mathcal{M}$-negligible variation outside of $G^{o}$, 
there exists a gauge $\delta_{v}$ on $J \cap (G^{o})^{c}$ 
such that for each $J \cap (G^{o})^{c}$-tagged 
$\delta_{v}$-fine $\mathcal{M}$-partition $\pi_{v}^{0}$ in $J$, 
we have
\begin{equation}\label{eq_Neg_Var1}
||\sum_{(t,I) \in \pi_{v}^{0}} H_{J}(I) ||  = 
||\sum_{(t,I) \in \pi_{v}^{0}} 
\left (  \sum_{k:|I \cap C_{k}|>0}  F(I \cap C_{k}) \right ) || 
< \frac{\varepsilon}{4}. 
\end{equation}
By Lemma \ref{L2.2}, we can choose $\delta_{v}$ so that 
for each $J \cap (G \setminus G^{o})$-tagged 
$\delta_{v}$-fine $\mathcal{M}$-partition $\pi_{v}$ in $J$, we have
\begin{equation}\label{eq_Neg_Var2}
||\sum_{(t,I) \in \pi_{v}} ( h_{J}(t)|I| - H_{J}(I) ) || = 
|| \sum_{(t,I) \in \pi_{v}} 
\left ( ~f(t)|I| - \sum_{k:|I \cap C_{k}| >0} F(I \cap C_{k})~ \right ) 
|| < \frac{\varepsilon}{2}.
\end{equation}

Define a gauge $\delta_{J} : J \to (0,+\infty)$ as follows.  
$$
\delta_{J}(t) 
= 
\left \{
\begin{array}{ll}
\delta_{0}(t) & \text{ if }t \in J \cap G^{o}  \\
\delta_{v}(t) & \text{ if }t \in J \cap (G^{o})^{c}.
\end{array}
\right.
$$
Let $\pi$ be an arbitrary $\delta_{J}$-fine $\mathcal{M}$-partition of $J$. 
Then, 
$$
\pi = \pi_{1} \cup \pi_{2} \cup \pi_{3},
$$ 
where 
\begin{equation*}
\begin{split}
\pi_{1} 
&= \{ (t,I) \in \pi : t \in J \cap G^{o} \}, \\
\pi_{2} 
&= \{(t,I) \in \pi : t \in J \cap (G \setminus G^{o})  \}, \\
\pi_{3} 
&= \{(t,I) \in \pi : t \in J \setminus G  \}. 
\end{split}
\end{equation*}
Note that    
\begin{equation*}
\begin{split}
||\sum_{(t,I) \in \pi} (~h_{J}(t)|I| - H_{J}(I)~) || 
\leq&
||\sum_{(t,I) \in \pi_{1}} (~h_{J}(t)|I| - H_{J}(I)~) || \\
+& ||\sum_{(t,I) \in \pi_{2}} (~h_{J}(t)|I| - H_{J}(I)~) || \\
+& ||\sum_{(t,I) \in \pi_{3}} H_{J}(I)||.
\end{split}
\end{equation*} 
Hence, by \eqref{eq_T21.1}, \eqref{eq_Neg_Var1} and \eqref{eq_Neg_Var2}, 
it follows that
\begin{equation*} 
\begin{split}
||\sum_{(t,I) \in \pi} (~h_{J}(t)|I| - H_{J}(I)~) || 
< \frac{\varepsilon}{4} + \frac{\varepsilon}{2} + \frac{\varepsilon}{4} = \varepsilon,  
\end{split}
\end{equation*} 
and since $\pi$ was an arbitrary $\delta_{J}$-fine $\mathcal{M}$-partition of $J$, 
we obtain that $h_{J}$ is McShane integrable on $J$ 
with the primitive $F_{J}$.

$(ii) \Rightarrow (i)$ 
Assume that $(ii)$ holds.   
Then, 
by Lemma \ref{L2.4}, 
given $\varepsilon >0$ there exists a gauge $\delta_{\varepsilon}$ on $\mathbb{R}^{m}$ such that  
for each $\delta_{\varepsilon}$-fine $\mathcal{M}$-partition $\pi$ in 
$\mathbb{R}^{m}$, 
we have
\begin{equation}\label{eq_LOCint.1}
|| \sum_{(t,I) \in \pi } 
(~ h(t)|I| - H(I)~) || < \varepsilon. 
\end{equation}
We can choose $\delta_{\varepsilon}$ so that 
$B_{m}(t, \delta_{\varepsilon}(t)) \subset G^{o}$ 
for all $t \in G^{o}$.  
Let $\pi_{0}$ be 
a $\delta_{0}$-fine $\mathcal{M}$-partition in $G^{0}$, 
where $\delta_{0} = \delta_{\varepsilon} \vert_{G^{o}}$. 
Then, $\pi_{0}$ is a $\delta_{\varepsilon}$-fine $\mathcal{M}$-partition in 
$\mathbb{R}^{m}$. 
Hence, by \eqref{eq_LOCint.1} it follows that
\begin{equation*}
|| \sum_{(t,I) \in \pi_{0} } 
(~ f(t)|I| - F(I)~) || < \varepsilon. 
\end{equation*}

By equality 
$$
H(I) = F(I), \text{ for all }I \in \mathcal{I}_{G}
$$ 
and by \eqref{eq_T21.0}, it follows that $F$ is a Hake-function.

It remains to prove that $F$ has $\mathcal{M}$-negligible variation outside of $G^{o}$. 
By Lemma \ref{L2.5}
there exists a gauge $\delta_{v}$ on $(G^{o})^{c}$ 
such that for each 
$\delta_{v}$-fine  $(G^{o})^{c}$-tagged $\mathcal{M}$-partition $\pi_{v}$ in $\mathbb{R}^{m}$, 
we have
\begin{equation*}
||\sum_{(t,I) \in \pi_{v}} H(I) || < \varepsilon.
\end{equation*}
Hence,
\begin{equation*}
||\sum_{(t,I) \in \pi_{v}} 
\left ( \sum_{k:|I \cap C_{k}|>0} F(I \cap C_{k}) \right) || \leq   \varepsilon, 
\end{equation*}
whenever $(C_{k}) \in \mathscr{D}_{G^{o}}$. 
This means that $F$ has $\mathcal{M}$-negligible variation outside of $G^{o}$, 
and this ends the proof.
\end{proof}

\begin{theorem}\label{T2.2}
Let $f: G \to X$ be a function and let $F: \mathcal{I}_{G} \to X$ be an additive interval function.  
Then, the following statements are equivalent:
\begin{itemize}
\item[(i)]
$f$ is Hake-Henstock-Kurzweil integrable on $G$ with the primitive $F$,
\item[(ii)] 
$h$ is locally Henstock-Kurzweil integrable on $\mathbb{R}^{m}$ with the primitive $H$ 
such that $H(I) = F(I)$, for all $I \in \mathcal{I}_{G}$, and
\begin{equation*}
H(I) = \sum_{k:|I \cap C_{k}|>0} F(I \cap C_{k}), \text{ for all }I \in \mathcal{I},
\end{equation*}
whenever $(C_{k}) \in \mathscr{D}_{G^{o}}$. 
\end{itemize}
\end{theorem}
\begin{proof}
In the same manner as the proof of Theorem \ref{T2.1}, 
by using Lemma \ref{L2.3}, it can be proved that $(i) \Rightarrow (ii)$ 
and, by using Lemmas \ref{L22.4} and \ref{L2.55},  
it can be proved that $(ii) \Rightarrow (i)$.
\end{proof}

\bibliographystyle{plain}

\end{document}